\documentclass[final,leqno,onefignum,onetabnum]{siamltex1213}

\usepackage{tikz}
\usetikzlibrary{arrows,chains,matrix,positioning,scopes}
\makeatletter
\tikzset{join/.code=\tikzset{after node path={%
\ifx\tikzchainprevious\pgfutil@empty\else(\tikzchainprevious)%
edge[every join]#1(\tikzchaincurrent)\fi}}}
\makeatother
\tikzset{>=stealth',every on chain/.append style={join},
         every join/.style={->}}
\tikzstyle{labeled}=[execute at begin node=$\scriptstyle,
   execute at end node=$]

\usepackage[hyphenbreaks]{breakurl}
\usepackage{cite}
\usepackage{amsmath, amssymb, amsfonts}
\usepackage{array}

\allowdisplaybreaks

\usepackage{rotating}
\usepackage{tablefootnote}

\usepackage{algorithm2e}
\usepackage[noend]{algpseudocode}

\usepackage[multiple]{footmisc}

\title{On the Relationship Between Real and Complex Linear Systems}

\author{C\'edric Josz\footnotemark[1] }

\usepackage{soul}

\begin{document}
\maketitle

\renewcommand{\thefootnote}{\fnsymbol{footnote}}

\footnotetext[1]{Laboratory for Analysis and Architecture of Systems (LAAS), French National Center for Scientific Research (CNRS), 7, avenue du Colonel Roche, Toulouse, 31000, France (\email{cedric.josz@gmail.com}). The research was funded by the European Research Council (ERC) under the European Union's Horizon 2020 research and innovation program (grant agreement 666981 TAMING).}

\renewcommand{\thefootnote}{\arabic{footnote}}

\slugger{mms}{xxxx}{xx}{x}{x--x}

\begin{abstract}
We consider the problem of solving a linear system of equations which involves complex variables and their conjugates. We characterize when it reduces to a complex linear system, that is, a system involving only complex variables (and not their conjugates). In that case, we show how to construct the complex linear system. Interestingly, this provides a new insight on the relationship between real and complex linear systems. In particular, any real symmetric linear system of equations can be solved via a complex linear system of equations. Numerical illustrations are provided. The mathematics in this manuscript constitute an exciting interplay between Schur's complement, Cholesky's factorization, and Cauchy's interlace theorem.
\end{abstract}



\pagestyle{myheadings}
\thispagestyle{plain}
\markboth{C\'EDRIC JOSZ}{Real and Complex Linear Systems}
\section{Introduction}
Let $\mathbb{N}$, $\mathbb{R}$ and $\mathbb{C}$ respectively denote the set of natural, real and complex numbers, and let $i$ denote the imaginary number satisfying $i^2=-1$. Let $\Re z$, $\Im z$, and $\bar{z}$ respectively denote the real part, the imaginary part, and the complex conjugate of $z \in \mathbb{C}$. Let $\mathcal{M}_n(\mathbb{R})$ and $\mathcal{M}_n(\mathbb{C})$ respectively denote the set of real-valued and complex-valued square matrices of size $n \in \mathbb{N}$. Consider $M,N \in \mathcal{M}_n(\mathbb{C})$ and $p \in \mathbb{C}^n$ and the system
\begin{equation}
\label{eq:system}
Mz + N\bar{z} = p~,~~~ z \in \mathbb{C}^n.
\end{equation}
The set of solutions to \eqref{eq:system} is a real affine space, but not a complex affine space in general. For instance, if $n=1$, $M=N=1$, and $p=0$, then the system reads $z + \bar{z} = 0,~z\in \mathbb{C}$. 
We raise the question of when system \eqref{eq:system} can be reduced to a complex linear system 
\begin{equation}
\label{eq:complex}
A z = b , ~~ z \in \mathbb{C}^n
\end{equation}
where $A \in \mathcal{M}_n(\mathbb{C})$ and $b \in \mathbb{C}^n$, and if so, how $A$ and $b$ can be constructed. Naturally, the set of solutions to \eqref{eq:complex} is a complex affine space. Thus the system reduction can take place if and only if the set of solutions of system \eqref{eq:system} is a complex affine space. 

The paper is organized as follows. Section \ref{sec:Reduction to complex linear system} contains the main results which are applied to real linear systems in Section \ref{sec:Application}. Section \ref{sec:Conclusion} concludes our work.
\section{Reduction to complex linear system}
\label{sec:Reduction to complex linear system}
When system \eqref{eq:system} has no solution, it can be reduced to a complex linear system \eqref{eq:complex} with $A = 0$ and $b = p$. When it has one solution or more, Proposition \ref{prop:reduction} provides a constructive reduction. We thus introduce the notation $\mathcal{R}(M,N)$ to denote the range of system \eqref{eq:system}, that is to say
\begin{equation}
\mathcal{R}(M,N) := \{ p \in \mathbb{C}^n ~|~ \exists z \in \mathbb{C}^n : Mz + N\bar{z} = p \}.
\end{equation}
Let's equip $\mathbb{C}^{2n}$ with Hermitian inner product $\langle \cdot, \cdot \rangle$ defined by $\langle u , v \rangle = \sum\limits_{k=1}^{2n} u_k \overline{v}_k$. Let $\mathcal{R}$ and $\text{Ker}$ respectively denote the range and Kernel of a linear application. The key ingredient to obtain the reduction is the direct sum
\begin{equation}
\label{eq:direct}
\mathbb{C}^{2n} = \mathcal{R} \begin{pmatrix} N \\[.1cm] \overline{M} \end{pmatrix} \oplus \left[  \mathcal{R} \begin{pmatrix} N \\[.1cm] \overline{M} \end{pmatrix} \right]^{\perp}
\end{equation}
and, in particular, the orthogonal projection $P_{\perp}$ onto the latter set in the direct sum. 
\begin{proposition}
\label{prop:reduction}
Let $n \in \mathbb{N}$ and $M,N \in \mathcal{M}_n(\mathbb{C})$. The following are equivalent:
\begin{enumerate}
\item $\{ z \in \mathbb{C}^n ~|~ Mz + N \bar{z} = 0 \}$ is a complex vector space.
\item The application defined from $\mathbb{C}^n$ to $\mathbb{C}^{2n}$ defined by $z \longmapsto P_{\perp} \begin{pmatrix} z\\ \bar{z} \end{pmatrix}$ is injective.
\item $\exists U,V \in \mathcal{M}_n(\mathbb{C}): ~ \forall p \in \mathcal{R}(M,N),$
\begin{equation}
\begin{array}{c}
\{ z \in \mathbb{C}^n ~|~ Mz + N \bar{z} = p \} \\[.2cm]
= \\[.2cm]
\{ z \in \mathbb{C}^n ~~|~~ (UM + V\overline{N})z = Up + V\overline{p} ~ \}.
\end{array}
\end{equation}
\end{enumerate}
\end{proposition}
\begin{proof}
The implication $(3 \Longrightarrow 1)$ follows from taking $p=0$. As for $(1 \Longrightarrow 2)$, consider $z \in \mathbb{C}^n$ such that 
\begin{equation}
P_{\perp} \begin{pmatrix} z\\ \bar{z} \end{pmatrix} = 0.
\end{equation}
Thus 
\begin{equation}
\begin{pmatrix} z\\[.1cm] \bar{z} \end{pmatrix} \in \mathcal{R} \begin{pmatrix} N \\[.1cm] \overline{M} \end{pmatrix}.
\end{equation}
As a consequence, there exists $w \in \mathbb{C}^n$ such that 
\begin{equation}
\label{eq:co}
\begin{pmatrix}
z \\[.1cm]
\bar{z}
\end{pmatrix}
=
\begin{pmatrix}
 Nw \\[.1cm]
\overline{M}w
\end{pmatrix}
\end{equation}
The first line minus the conjugate of the second line yields 
\begin{equation}
\label{eq:cool}
Nw - M \overline{w} = 0.
\end{equation} Multiplying by the imaginary number, we get $M(\overline{iw}) + N(iw) = 0$. Thus $\overline{iw}$ belongs to $\{ z \in \mathbb{C}^n ~|~ Mz + N \bar{z} = 0 \}$, which is a complex vector space thanks to Point 1. Thus $i(\overline{iw}) = \overline{w}$ also belongs to it, so that $M\overline{w} + Nw = 0$. Together with \eqref{eq:cool}, this implies that $Nw = M\overline{w} = 0$. Going back to \eqref{eq:co}, it holds that $z = 0$.

We now prove $(2 \Longrightarrow 3)$. Let $p \in \mathcal{R}(M,N)$ and consider $z\in \mathbb{C}^n$. Point 2 implies that the following four equations are equivalent:
\begin{equation}
Mz+N\bar{z} = p 
\end{equation}
\begin{equation}
\begin{pmatrix} Mz \\[.1cm] \overline{N}z \end{pmatrix} + \begin{pmatrix} N \bar{z} \\[.1cm] \overline{M} \bar{z} \end{pmatrix} =  \begin{pmatrix} p \\[.15cm] \overline{p} \end{pmatrix}
\end{equation}
\begin{equation}
P_{\perp} \begin{pmatrix} Mz \\[.1cm] \overline{N}z \end{pmatrix} + P_{\perp}\begin{pmatrix} N \bar{z} \\[.1cm] \overline{M} \bar{z} \end{pmatrix} =  P_{\perp}\begin{pmatrix} p \\[.15cm] \overline{p} \end{pmatrix}
\end{equation}
\begin{equation}
\label{eq:proj}
P_{\perp} \begin{pmatrix} Mz \\[.1cm] \overline{N}z \end{pmatrix} = P_{\perp}\begin{pmatrix} p \\[.15cm] \overline{p} \end{pmatrix}
\end{equation}
since, by definition of $P_{\perp}$, we have
\begin{equation}
P_{\perp}\begin{pmatrix} N \bar{z} \\[.1cm] \overline{M} \bar{z} \end{pmatrix} = 0.
\end{equation}
Let $\begin{pmatrix} \hat{U} ~ \hat{V} \end{pmatrix}$ denote the matrix representation of $P_{\perp}$ in the canonical basis of $\mathbb{C}^{2n}$. The complex matrices $\hat{U}$ and $\hat{V}$ have $2n$ rows and $n$ columns. Equation \eqref{eq:proj} is thus equivalent to
\begin{equation}
\label{eq:hat}
(\hat{U}M + \hat{V}\overline{N})z = \hat{U}p + \hat{V}\overline{p}.
\end{equation}
The matrix $\hat{U}M+\hat{V}\overline{N}$ has $2n$ rows and $n$ columns so the above system has $2n$ equations and $n$ unknowns. We would like to obtain a system with $n$ equations. First, note that it has at least one solution because $p \in \mathcal{R}(M,N)$. Second, a linear system of equations with more equations than unknowns either has no solutions, or any equation that is a linear combination of other equations can be removed without changing the set of solutions. Thus, consider a set $I \subset \{ 1 , \hdots , 2n \}$ of $n$ rows that generate the row space of $\hat{U}M+\hat{V}\overline{N}$. Define $U,V \in \mathcal{M}_n(\mathbb{C})$ such that
\begin{equation}
\begin{array}{c}
\forall i \in I,~ \forall j \in \{1,\hdots,n\}, ~ U_{ij} := \hat{U}_{ij},~ V_{ij} := \hat{V}_{ij}.
\end{array}
\end{equation}
We stress that $U$ and $V$ do not depend on the vector $p$, but only on the matrices $M$ and $N$. We now decude that \eqref{eq:hat} is equivalent to $(UM + V\overline{N})z = Up + V\overline{p}$, which terminates the proof.
\end{proof}
Point 2 in Proposition \ref{prop:reduction} provides a necessary and sufficient condition on $M$ and $N$ to determine whether system \eqref{eq:system} can be reduced to a system of the form \eqref{eq:complex}. Let's write the condition more explicitly. The projection $P_{\perp}$ is an endormorphism of $\mathbb{C}^{2n}$, so let's consider its block decomposition 
\begin{equation}
P_{\perp} = \begin{pmatrix} P_{11} & P_{12} \\ P_{21} & P_{22} \end{pmatrix}
\end{equation}
where $P_{11}, P_{12}, P_{11}, P_{11}$ are endormorphisms of $\mathbb{C}^{n}$. For all $z \in \mathbb{C}^n$, we then have
\begin{equation}
P_{\perp} \begin{pmatrix} z\\ \bar{z} \end{pmatrix} = \begin{pmatrix} P_{11} & P_{12} \\ P_{21} & P_{22} \end{pmatrix} \begin{pmatrix} x+iy\\ x-iy \end{pmatrix} = \begin{pmatrix} P_{11}+P_{12} & i(P_{11}-P_{12}) \\ P_{21}+P_{22} & i(P_{21}-P_{22}) \end{pmatrix} \begin{pmatrix} x\\ y \end{pmatrix}
\end{equation}
where $x := \Re z$ and $y := \Im z$. Point 2 is thus equivalent to
\begin{equation}
\label{eq:big}
\text{Ker} \begin{pmatrix} \Re P_{11}+ \Re P_{12} & \Im P_{12} - \Im P_{11} \\ \Re P_{21}+ \Re P_{22} &  \Im P_{22} - \Im P_{21} \\ 
\Im P_{11}+ \Im P_{12} & \Re P_{11}- \Re P_{12} \\ \Im P_{21}+ \Im P_{22} & \Re P_{21}- \Re P_{22} \end{pmatrix} = \{0\}
\end{equation}
where the domain of the application in \eqref{eq:big} is $\mathbb{C}^{4n}$ and its co-domain is $\mathbb{C}^{2n}$.\\\\
We illustrate Proposition \ref{prop:reduction} with the following example:
\begin{equation}
M := \begin{pmatrix}
0 & -i & 0 & 2-i & 5i \\
0 & 3i & 0 & 3 & 9i \\
-1 & 5 & 1-3i & -3+3i & 1+7i \\
0 & -2i & 0 & -i & -i \\
0 & 1-i & 0 & i &  -2-3i
\end{pmatrix},
\end{equation}
\begin{equation}
N := \begin{pmatrix}
-i & 5i & -3+i & 3-3i & 7+i \\
0 & -3 + 2i & 0 & -2 & -1+3i  \\
0 & -1 & 0 & -1+2i & 5 \\
0 & 1 & 0 & i & -i \\
0 & 5 & 0 & 7+4i & 1+2i
\end{pmatrix},
\end{equation}
\begin{equation}
p := \begin{pmatrix}
1-i \\
3 \\
-1+i \\
5+i \\
1
\end{pmatrix}.
\end{equation}
Using a reduced column echelon form and a Grahm-Schmidt orthogonalization, we compute the matrix associated to $P_{\perp}$ in the canonical basis of $\mathbb{C}^{2n}$. We then check whether the system can be reduced using \eqref{eq:big}. In this example, reduction is possible. We then compute a set of 5 rows $I = \{3,4,5,6,7\}$ that generate the row space of \begin{equation} P_{\perp} \begin{pmatrix} M \\[.1cm] \overline{N} \end{pmatrix}.\end{equation} We then deduce
$$
\small
U := 
\begin{pmatrix}
   0.0000 + 0.0000i & -0.1686 - 0.0351i &  \hphantom{-}0.6782 + 0.0000i &  \hphantom{-}0.0612 + 0.0175i & -0.0133 - 0.0054i \\
   0.0000 + 0.0000i &  \hphantom{-}0.1592 - 0.0009i  & \hphantom{-}0.0612 - 0.0175i  & \hphantom{-}0.9171 + 0.0000i &  \hphantom{-}0.0044 - 0.1532i \\ 
   0.0000 + 0.0000i &  \hphantom{-}0.0309 - 0.0069i & -0.0133 + 0.0054i  & \hphantom{-}0.0044 + 0.1532i &  \hphantom{-}0.0329 + 0.0000i \\
   0.0000 + 0.0000i &  -0.0351 + 0.1686i &  \hphantom{-}0.0000 + 0.3218i &  \hphantom{-}0.0175 - 0.0612i & -0.0054 + 0.0133i \\
   0.0000 + 0.0000i  & \hphantom{-}0.2393 - 0.1929i  & -0.0152 + 0.1880i  & -0.0881 + 0.0651i  & \hphantom{-}0.0169 + 0.0212i
\end{pmatrix}
$$
and
$$
\small
V:=
\begin{pmatrix}
		\hphantom{-}0.0000 - 0.3218i
	&	-0.0152 - 0.1880i
	&	0.0000 + 0.0000i
	&	-0.0128 - 0.1089i
	&	\hphantom{-}0.1766 + 0.0448i
	\\
\hphantom{-}0.0175 + 0.0612i & -0.0881 - 0.0651i & 0.0000 + 0.0000i &  \hphantom{-}0.0510 + 0.0511i  & \hphantom{-}0.0000 + 0.0433i \\
-0.0054 - 0.0133i  & \hphantom{-}0.0169 - 0.0212i  & 0.0000 + 0.0000i  &-0.0523 + 0.0314i & -0.0300 - 0.0397i \\
 \hphantom{-}0.6782 + 0.0000i & -0.1880 + 0.0152i  & 0.0000 + 0.0000i & -0.1089 + 0.0128i &  \hphantom{-}0.0448 - 0.1766i \\
 -0.1880 - 0.0152i  & \hphantom{-}0.4239 + 0.0000i  & 0.0000 + 0.0000i  & \hphantom{-}0.1337 - 0.0184i  & \hphantom{-}0.2104 - 0.0583i
\end{pmatrix}.
$$
Next, we compute the reduced complex linear system \eqref{eq:complex} where
$$
\small
A:=
\begin{pmatrix}
 -0.3563 + 0.0000i &  \hphantom{-}2.4429 + 0.0891i  &  \hphantom{-}0.3563 - 1.0690i & -0.2151 + 0.8851i &  \hphantom{-}0.5264 + 0.8788i \\
  -0.1225 + 0.0349i & \hphantom{-}0.6516 - 1.0495i &  \hphantom{-}0.0178 - 0.4024i &  \hphantom{-}0.7686 - 0.0611i & -0.1645 + 2.0554i \\
   \hphantom{-}0.0266 - 0.0107i  & -0.0681 - 0.0324i &  \hphantom{-}0.0057 + 0.0905i & -0.0774 - 0.1673i  & -0.1734 - 0.0619i \\
   \hphantom{-}0.0000 + 0.3563i &  \hphantom{-}0.0891 - 2.4429i  & -1.0690 - 0.3563i  & \hphantom{-}0.8851 + 0.2151i  & \hphantom{-}0.8788 - 0.5264i \\ 
   \hphantom{-}0.0303 - 0.3759i &  \hphantom{-}0.5094 + 1.6202i  & \hphantom{-}1.0975 + 0.4669i & \hphantom{-}0.0983 - 3.0761i &  -1.1428 + 0.6953i
\end{pmatrix}
$$
and
$$
b :=
\begin{pmatrix}
  -0.6283 - 0.6565i \\
   \hphantom{-}5.0216 + 0.9709i \\
  -0.1995 + 0.8183i \\
  -0.6565 + 0.6283i \\
   \hphantom{-}2.0157 - 1.0104i
\end{pmatrix}.
$$
We can now compute the solutions to the reduced system:
\begin{equation}
\begin{pmatrix} 
 -27.4310 +50.9483i \\
  -4.0647 + 5.7543i \\
   \hphantom{-.}0.0000 + 0.0000i \\
   \hphantom{-.}2.7694 - 1.2220i \\
   \hphantom{-.}0.6875 + 0.9203i
\end{pmatrix}
+ 
\mathbb{C}
\begin{pmatrix} 
1-3i \\
0 \\
1 \\
0 \\
0
\end{pmatrix}
\end{equation}

 We then arbitrarily chose one solution to the reduced system, and then check whether it satisfies the original system. If so, then the set of solutions to the original system and to the reduced system coincide. If not, then the original system has no solutions. In this example, the original system and the reduced system coincide.\\

We now turn our attention to the special case when the system \eqref{eq:system} has a unique solution. When in addition $M$ or $N$ is invertible, there exists an analytic expression for the matrices $U$ and $V$ in Proposition \ref{prop:reduction} that define the reduction to a complex linear system. Below, the matrix $I$ denotes the identity of $\mathcal{M}_n(\mathbb{C})$.
\begin{proposition}
\label{prop:system}
Let $n \in \mathbb{N}$ and $M,N \in \mathcal{M}_n(\mathbb{C})$.
\begin{equation}
\label{eq:inv}
\begin{array}{c}
\forall p \in \mathbb{C}^n, ~ \exists ! z \in \mathbb{C}^n: ~ Mz + N \bar{z} = p \\[.2cm] \Longleftrightarrow \\[.2cm] 
\begin{pmatrix}
M & N \\[.1cm]
\overline{N} & \overline{M} 
\end{pmatrix} ~\text{is invertible.}
\end{array}
\end{equation}
Assume that either equivalent property holds in \eqref{eq:inv}. Then, for all $p \in \mathbb{C}^n$,
\begin{equation}
\begin{array}{c}
\{ z \in \mathbb{C}^n ~|~ Mz + N \bar{z} = p \} \\[.2cm]
= \\[.2cm]
\{ z \in \mathbb{C}^n ~~|~~ (UM + V\overline{N})z = Up + V\overline{p} ~ \}
\end{array}
\end{equation}
where $U := I$ and $V:= -N\overline{M}^{-1}$ if $M$ is invertible, and where $U := N^{-1}$ and $V := -\overline{M}^{-1}$ if $M$ and $N$ are invertible.
\end{proposition}
\begin{proof}
($\Longrightarrow$) Consider $z,w \in \mathbb{C}^n$ such that
\begin{equation}
\begin{pmatrix}
M & N \\[.1cm]
\overline{N} & \overline{M} 
\end{pmatrix}
\begin{pmatrix}
z\\[.2cm]
\overline{w} 
\end{pmatrix}
=
\begin{pmatrix}
0 \\[.2cm]
0
\end{pmatrix}.
\end{equation}
This implies after conjugation that
\begin{equation}
\begin{pmatrix}
M & N \\[.1cm]
\overline{N} & \overline{M} 
\end{pmatrix}
\begin{pmatrix}
w  \\[.2cm]
\overline{z}
\end{pmatrix}
=
\begin{pmatrix}
0 \\[.2cm]
0
\end{pmatrix}.
\end{equation}
Adding the two previous equations yields
\begin{equation}
\begin{pmatrix}
M & N \\[.1cm]
\overline{N} & \overline{M} 
\end{pmatrix}
\begin{pmatrix}
z+w  \\[.2cm]
\overline{z+w}
\end{pmatrix}
=
\begin{pmatrix}
0 \\[.2cm]
0
\end{pmatrix}.
\end{equation}
Thus $z+w$ is a solution to system \eqref{eq:system} for $p=0$, whose unique solution is $0$. Hence $z+w = 0$. As a result, $Mz - N\bar{z} = 0$, that is to say $M(iz) + N(\overline{iz}) = 0$. Again by unicity of the solution, $i z = 0$. To conclude, $z = w = 0$. \\
($\Longleftarrow$)
\textit{Existence of a solution:} there exists $z,w \in \mathbb{C}^n$ such that
\begin{equation}
\begin{pmatrix}
M & N \\[.1cm]
\overline{N} & \overline{M} 
\end{pmatrix}
\begin{pmatrix}
z\\[.2cm]
\overline{w}
\end{pmatrix}
=
\begin{pmatrix}
p \\[.2cm]
\overline{p}
\end{pmatrix}
\end{equation}
which implies after conjugation that
\begin{equation}
\begin{pmatrix}
M & N \\[.1cm]
\overline{N} & \overline{M} 
\end{pmatrix}
\begin{pmatrix}
w\\[.2cm]
\overline{z}
\end{pmatrix}
=
\begin{pmatrix}
p \\[.2cm]
\overline{p}
\end{pmatrix}
\end{equation}
Since the system has a unique solution, it must be that $z=w$. Thus $z$ is a solution to \eqref{eq:system}.\\
\textit{Unicity of the solution:} consider $z,w \in \mathbb{C}^n$ such that $Mz +N\bar{z} = p$ and $Mw +N\overline{w} = p$. It must be that $M(z-w) + N\overline{(z-w)} = 0$, thus
\begin{equation}
\begin{pmatrix}
M & N \\[.1cm]
\overline{N} & \overline{M} 
\end{pmatrix}
\begin{pmatrix}
z-w\\[.2cm]
\overline{z-w}
\end{pmatrix}
=
\begin{pmatrix}
0 \\[.2cm]
0
\end{pmatrix}.
\end{equation}
Since the above block matrix is invertible, it follows that $z = w$.\\\\
To find suitable choices for the matrices $U$ and $V$, it suffices for them to satisfy that $UN + V \overline{M} = 0$ and that the application from $\mathbb{C}^n$ to itself defined by
\begin{equation}
z \longmapsto Uz + V\bar{z}
\end{equation}
is injective.
Indeed, in that case, we have the following equivalences for $z\in \mathbb{C}^n$:
\begin{equation}
Mz+N\bar{z} = p 
\end{equation}
\begin{equation}
\begin{pmatrix} Mz \\[.1cm] \overline{N}z \end{pmatrix} + \begin{pmatrix} N \bar{z} \\[.1cm] \overline{M} \bar{z} \end{pmatrix} =  \begin{pmatrix} p \\[.1cm] \overline{p} \end{pmatrix}
\end{equation}
\begin{equation}
(UM + V\overline{N})z + (UN+V\overline{M})\bar{z} = Up + V\overline{p}
\end{equation}
\begin{equation}
(UM + V\overline{N})z = Up + V\overline{p}.
\end{equation}
Let's assume that $M$ is invertible. The relationship $UN + V \overline{M} = 0$ is easy to prove. As for the injectivity, consider $z \in \mathbb{C}^n$ such that $z -N\overline{M}^{-1}\bar{z} = 0$. Then $w := M^{-1}z$ satisfies $Mw - N\overline{w} = 0$, so that $M(iw) + N(\overline{iw}) = 0$. By assumption, the unique solution is $iw = 0$, hence $z =0$. To treat the case where $M$ and $N$ are invertible, it suffices to multiply $U$ and $V$ by $N^{-1}$.
\end{proof}

\section{Application to real linear system}
\label{sec:Application}
Consider $m \in \mathbb{N}$, a matrix $A \in \mathcal{M}_m(\mathbb{R})$ that is positive definite, a vector $b \in \mathbb{R}^m$, and the linear system
\begin{equation}
\label{eq:linear}
A x = b , ~ x \in \mathbb{R}^m.
\end{equation}
Thanks to a Gauss pivot, we may assume that $m=2n$ and view the system over the set of complex numbers. The first half of the variables can be viewed as real parts of complex variables, and the second half as their imaginary parts. To this end, define $B,C,D \in \mathcal{M}_n(\mathbb{R})$ and $z,p \in \mathbb{C}^n$ such that
$$
 A =:
\begin{pmatrix}
B\hphantom{^T} & C \\
C^T & D
\end{pmatrix}
,~~
 z := \begin{pmatrix} x_1 \\ \vdots \\ x_n \end{pmatrix} + i \begin{pmatrix} x_{n+1} \\ \vdots \\ x_{2n} \end{pmatrix} ,~~
p := \begin{pmatrix} b_1 \\ \vdots \\ b_n \end{pmatrix} + i \begin{pmatrix} b_{n+1} \\ \vdots \\ b_{2n} \end{pmatrix}, 
$$
where $(\cdot)^T$ stands for transpose. System \eqref{eq:linear} is equivalent to
\begin{equation}
\label{eq:conjugate}
M z + N \bar{z} = p~ , ~~  z \in \mathbb{C}^n,
\end{equation}
where $M := (B+D-i(C-C^T))/2  \succ 0 $ (positive definiteness will become clear in the following) and $N := (B-D+i(C+C^T))/2$. Indeed, system \eqref{eq:linear} can be written as either of the three following equivalent systems:
\begin{enumerate}
\item $ 
\frac{1}{2}\begin{pmatrix}
B & C \\
C^T & D
\end{pmatrix}
\begin{pmatrix}
z+\bar{z} \\
i(\bar{z}-z)
\end{pmatrix}
=
b, $\\\\
\item $\left\{ \begin{array}{rcl}
\frac{B-iC}{2}z + \frac{B+iC}{2}\bar{z} & = & \Re p \\[.5em]
\frac{C^T-iD}{2}z + \frac{C^T+iD}{2}\bar{z} & = & \Im p, \\
\end{array}
\right.$
\\\\
\item $\frac{B+D-i(C-C^T)}{2} z + \frac{B-D+i(C+C^T)}{2} \bar{z} = p.$
\end{enumerate}
Proposition \ref{prop:system} informs us that system \eqref{eq:conjugate} is in turn equivalent to
\begin{equation}
\label{eq:embedding}
(M - N\overline{M}^{-1}\overline{N})z = p - N\overline{M}^{-1}\overline{p}.
\end{equation}
As a result, the real linear system \eqref{eq:linear} is equivalent to the complex linear system \eqref{eq:embedding}. To the best of our knowledge, this relationship between real and complex linear systems has not been presented in past literature. 

With regards to numerical aspects, an efficient method for solving the linear system \eqref{eq:linear} is the Cholesky factorization~\cite{cholesky1910} which requires roughly $m^3/3 = 8n^2/3$ flops. It searches for a lower triangular matrix $L \in \mathbb{R}^{m \times m}$ such that $A = LL^T$. Next, it successively solves for $Ly=b, ~ y \in \mathbb{R}^m$ and $L^Tx = y,~ x \in \mathbb{R}^m$. Unfortunately, the matrix multiplication in the complex linear system \eqref{eq:embedding} makes it more expensive to solve. However, the spectra of the positive definite matrices $\overline{M} \in \mathcal{M}_n(\mathbb{C})$ and $M - N\overline{M}^{-1}\overline{N} \in \mathcal{M}_n(\mathbb{C})$ are contractions of the spectrum of $A \in \mathcal{M}_{2n}(\mathbb{R})$. They hence have a better conditioning number (ratio of maximum to minimum eigenvalue) than $A$. This is relevant since a Cholesky factorization of both these matrices is required to solve the complex linear system \eqref{eq:embedding}. For work on the complex Cholesky factorization, see \cite{bereux2003,bereux2005}. 

We now clarify our statement regarding spectrum contraction. Given a Hermitian matrix $U \in \mathcal{M}_n(\mathbb{C})$, let $\lambda_1(U), \hdots, \lambda_n(U)$ denote its eigenvalues in increasing order. Given $\lambda \in \mathbb{R}$ and $z,w \in \mathbb{C}^m$, it holds that
$$
\begin{pmatrix}
M & N \\[.1cm]
\overline{N} & \overline{M}
\end{pmatrix}
\begin{pmatrix}
z\\[.2cm]
\overline{w}
\end{pmatrix}
=
\lambda
\begin{pmatrix}
z\\[.2cm]
\overline{w}
\end{pmatrix}
\Longleftrightarrow
\begin{pmatrix}
B\hphantom{^T} & C \\[.1cm]
C^T  & D
\end{pmatrix}
\begin{pmatrix}
\Re (z+w) \\[.1cm]
\Im (z+w)
\end{pmatrix}
=
\lambda
\begin{pmatrix}
\Re (z+w) \\[.1cm]
\Im (z+w)
\end{pmatrix}
$$
thereby equating the spectra of the two square matrices in the equation. Since $M$ is an extraction of the left-hand matrix, Cauchy's interlace theorem \cite{cauchy,lancaster1985,hwang2004,fisk2005} implies that 
\begin{equation}
\lambda_k (A)
\leqslant 
\lambda_k(M) \leqslant 
\lambda_{k+n} (A)~ , ~~~ k = 1, \hdots , n.
\end{equation}
In addition, the interlacing property of the eigenvalues of the Schur complement of a Hermitian matrix \cite[Theorem 5]{smith1992} implies that 
\begin{equation}
\lambda_k (A)
\leqslant 
\lambda_k(M-N\overline{M}^{-1} \overline{N}) \leqslant 
\lambda_{k+n} (A) ~ , ~~~ k = 1, \hdots , n.
\end{equation}
See \cite{zhang2005} for a nice reference on the Schur complement.
We now provide a numerical experiment. Consider the real linear system \eqref{eq:linear} with $m := 6$,
\begin{equation}
A:=
\left(
\begin{array}{rrrrrr}
2.0483 &  -0.3065 &   0.7403  & -0.3338 &   0.9431  &  1.4834 \\
   -0.3065 &   1.2538 &  -1.1144 &   0.7319 &  -0.2412  &  0.1729 \\
    0.7403  & -1.1144  &  1.8337 &  -0.6019 &  -0.0518  &  0.6788 \\
   -0.3338  &  0.7319 &  -0.6019  &  1.6525 &   0.4313 &   0.0371 \\
    0.9431  & -0.2412 &  -0.0518 &   0.4313 &   1.4893  &  0.3625 \\
    1.4834  &  0.1729 &   0.6788  &  0.0371  &  0.3625  &  1.5775
\end{array}
\right)
\end{equation}
and 
\begin{equation}
b :=
\left(
\begin{array}{r}
   -1.7746 \\
   -1.3900 \\ 
   -1.9215 \\
   -0.2593 \\
    1.3289 \\
    0.4696
\end{array}
\right).
\end{equation}
The data has been randomly generated with MATLAB R2015b. The eigenvalues of $A$ in increasing order are 
\begin{equation}
\label{eq:A}
\lambda(A) =
\left(
\begin{array}{r}
    0.0245\\
    0.1082\\
    1.0932\\
    1.4319\\
    2.8521\\
    4.3453
\end{array}
\right)
\end{equation}
MATLAB R2015b yields the following solution to $Ax = b, ~ x \in \mathbb{R}^m$:
\begin{equation}
\label{eq:x}
x =
\left(
\begin{array}{r}
    -33.1807\\
  -56.9574\\
  -42.5687\\
    2.4589\\
   -3.3323\\
   56.7669
\end{array}
\right)
\end{equation}
Based on the discussion in this section, we may instead solve the complex linear system $(M - N\overline{M}^{-1}\overline{N})z = p - N\overline{M}^{-1}\overline{p},~ z \in \mathbb{C}^n,$ where
\begin{equation}
M - N\overline{M}^{-1}\overline{N} =
\left(
\begin{array}{rrr}
0.5756 - 0.0000i &  0.3906 + 0.2060i  & -0.1576 - 0.5196i \\
   0.3906 - 0.2060i &  0.8131 - 0.0000i & -0.5220 - 0.6089i \\
  -0.1576 + 0.5196i  & -0.5220 + 0.6089i  & 0.9137 - 0.0000i
\end{array}
\right)
\end{equation}
and
\begin{equation}
p - N\overline{M}^{-1}\overline{p} =
\left(
\begin{array}{r}
0.1759 - 1.9830i\\
   2.1302 + 1.2597i\\
   1.4658 - 1.7835i 
\end{array}
\right).
\end{equation}
This yields
\begin{equation}
z =
\left(
\begin{array}{c}
    -33.1807 + 2.4589i\\
 -56.9574 - 3.3323i\\
 -42.5687 +56.7669i
\end{array}
\right)
\end{equation}
to be compared with $x$ in \eqref{eq:x}.

The eigenvalues of $M - N\overline{M}^{-1}\overline{N}$ in increasing order are
\begin{equation}
\lambda(M - N\overline{M}^{-1}\overline{N}) =
\left(
\begin{array}{r}
0.0463 \\
   0.2488 \\
   2.0073 
\end{array}
\right)
\end{equation}
which can be seen to interlace those of $A$ in \eqref{eq:A}. As a byproduct, the conditioning number of $M - N\overline{M}^{-1}\overline{N}$, equal to $43.3843$, is less than that of $A$, equal to $177.3795$. We remind the reader that the closer to 1, the better the conditioning of a matrix. 
\section{Conclusion}
\label{sec:Conclusion}
We classify linear systems involving complex variables and their conjugates with respect to systems involving only complex variables. Our main tool is the orthogonal projection. As a result, we obtain a new link between linear systems with real variables and linear systems with complex variables. Our findings are illustrated on several numerical examples. A natural question follows: when can algebraic varieties defined by non-holomorphic polynomials be reduced to varieties defined by holomorphic polynomials? This manuscript gives the answer when the polynomials are of degree one.

\section*{Acknowledgements}
I wish like to thank the anonymous reviewers for their helpful comments. Many thanks to Ximun Loyatho and Antoine Coutand for the fruitful discussions during my visit to Imswan.

\bibliography{mybib}{}
\bibliographystyle{siam}

\end{document}